\documentclass[a4paper]{article}
\usepackage[top=1in, bottom=1in, left=1.25in, right=1.25in]{geometry}

\usepackage[english]{babel}
\usepackage[utf8]{inputenc}
\usepackage{amsmath,amssymb,amsthm}
\usepackage{algorithm,algorithmic}
\usepackage{graphicx}
\usepackage[colorinlistoftodos]{todonotes}
\usepackage{multirow}

\newcommand{\bx}{{x}}

\newcommand{\bigO}{\mathcal{O}}
\newcommand{\norm}[2]{\lVert #1 \rVert_{#2}}
\newcommand{\normtwo}[1]{\lVert #1 \rVert_2}
\newcommand{\define}{\stackrel{\text{def}}{=}}
\newcommand{\at}[2]{\left.#1\right|_{#2}}
\newcommand{\Span}[1]{\text{span}\left\{#1\right\}}
\newcommand{\krylov}[2]{{\cal{K}}_{#2}(#1)}

\newcommand{\hphi}{\hat{\phi}}
\newcommand{\hpsi}{\hat{\psi}}

\newtheorem{propos}{Proposition}

\title{A fast algorithm for parabolic PDE-based inverse problems based on Laplace transforms and flexible Krylov solvers}
\author{Tania Bakhos \thanks{taniab@stanford.edu} \thanks{Institute for Computational and Mathematical Engineering, Stanford University} \and
Arvind K. Saibaba \thanks{Department of Electrical and Computer Engineering, Tufts University} \and Peter K. Kitanidis \thanks{Department of Civil and Environmental Engineering, Stanford University}}

\begin{document}
\maketitle

\begin{abstract}

We consider the problem of estimating parameters in large-scale weakly nonlinear inverse problems for which the underlying governing equations is a linear, time-dependent, parabolic partial differential equation. A major challenge in solving these inverse problems using Newton-type methods is the computational cost associated with solving the forward problem and with repeated construction of the Jacobian, which represents the sensitivity of the measurements to the unknown parameters. Forming the Jacobian can be prohibitively expensive because it requires repeated solutions of the forward and adjoint time-dependent parabolic partial differential equations corresponding to multiple sources and receivers. We propose  an efficient method based on a Laplace transform-based exponential time integrator combined with a flexible Krylov subspace approach to solve the resulting shifted systems of equations efficiently. Our proposed solver speeds up the computation of the forward and adjoint problems, thus yielding significant speedup in total inversion time. We consider an application from Transient Hydraulic Tomography (THT), which is an imaging technique to estimate hydraulic parameters related to the subsurface from pressure measurements obtained by a series of pumping tests. The algorithms discussed are applied to a synthetic example taken from THT to demonstrate the resulting computational gains of this proposed method.
\end{abstract}

\section{Introduction}
Consider linear parabolic partial differential equations (PDEs) of the form, 
\begin{equation}
\mathcal{S}(x;s)\frac{\partial \phi}{\partial t}  - {\mathcal{A}(x;s) }\phi   =  f(x,t) \qquad 
\phi (t=0)  = \phi_0 \label{eq:main}
\end{equation}
along with the associated boundary conditions, where $s$ are the parameters to be estimated, $f(x,t)$ represents the source term, $\phi(x,t)$ is the state variable and $\mathcal{S}$ and $\mathcal{A}$ are time-independent operators. We are interested in fast solutions to Equation~\eqref{eq:main} particularly in the context of inverse problems~\cite{tarantola2005inverse}, i.e. using measurements $\phi$ to estimate parameters $s$, whose governing forward equations are described by Equation~\eqref{eq:main}. Problems of this kind arise in many applications such as Diffuse Optical Tomography~\cite{boas2001imaging, kilmer2006recycling} ( $\mathcal{A} = \nabla \cdot( D \nabla \cdot) - \nu \mu_a$), electromagnetic inversion ($\mathcal{A} = \nabla \times \mu^{-1}\nabla \times $)~\cite{zaslavsky2012large} and Transient Hydraulic Tomography ($\mathcal{A} = \nabla \cdot( \kappa \nabla \cdot)$)~\cite{cardiff2013hydraulic}. For example, in Diffuse  Optical Tomography, measurements of photon fluence are used to ``invert'' for diffusivity $D$ and the absorption coefficient $\mu_a$.

We will focus our attention on Transient Hydraulic Tomography (THT). THT is a method of imaging of the subsurface that uses a series of pumping tests to estimate important hydrological parameters, such as conductivity and storativity. In this method, water is injected continuously at a constant rate in injection wells and the resulting response in pressure change is recorded at measurement wells, until steady-state is approximately reached. The drawdown curves recorded by each measurement location are stored and from this data, a few key time sampling points are identified and the corresponding measured pressure data are used in an inversion algorithm to recover the aquifer parameters of interest (e.g. hydraulic conductivity and specific storage). Instead of solving for the groundwater equations in the traditional time-stepping formulation, we propose a Laplace transform-based solver. We emphasize that while we describe methods as related to THT, they can be extended to a large class of linear parabolic PDEs.

We focus our attention on methods based on successive linearization such as Gauss-Newton to solve the inverse problem. However, we note that other possible approaches include Ensemble Kalman Filters, particle filters, rejection samplers, Markov Chain Monte Carlo, for a good review please see~\cite{kaipio2006statistical}.  Popular approaches for solving  inverse problems based on time-dependent parabolic PDEs rely on traditional time-stepping methods (such as Crank-Nicolson) that require the solutions of linear systems of equations at each time step in order to implicitly march the solution in time. The time step is computed based on accuracy and stability requirements. This process is inherently sequential and is computationally expensive. A second challenge in the context of inverse problems solved using Newton-type methods is that the computation of the Jacobian or sensitivity using the adjoint state method requires storing the entire time history of the forward problem. To alleviate the memory costs, previous work has considered checkpointing methods to trade computational efficiency for memory~\cite{griewank2000algorithm,symes2007reverse,wang2009minimal}. Therefore, there is a need to develop a method that is accurate and efficient, both in terms of memory and computational costs.

For solving the problem that arises in the discretization of the forward problem, previous works have noted that applying the contour integral representation of Laplace transforms leads to a sequence of shifted systems of linear equations which can then each be solved for independently. This observation can be exploited to develop a parallel algorithm for time integration for parabolic PDEs~\cite{sheen2000parallel}  or an iterative Richardson method~\cite{sheen2003parallel}. In~\cite{in2011contour}, the application of Krylov subspace methods that take advantage of the shift-invariance of the Krylov subspace to simultaneously solve the shifted systems arising from the contour integral representation was considered. However, to the best of our knowledge, the application of these methods to inverse problems has not been performed.

\textbf{Contributions:} In this paper, we propose a Laplace transform-based exponential time integrator combined with a flexible Krylov subspace approach to speed up the computation of the forward and adjoint problems, thus yielding significant speedup in total inversion time. The modified Talbot contour is chosen to invert the Laplace transform and is discretized using the Trapezoidal rule to yield an exponentially convergent quadrature rule~\cite{trefethen2006talbot}. The resulting shifted systems of equations are solved using a flexible Krylov subspace solver that allows for multiple preconditioners to efficiently precondition the systems across the entire range of shifts~\cite{saibaba2013flexible}. This yields significant speedup in solving the forward problem.  We propose heuristics to pick preconditioners and demonstrate the robustness of the proposed solver with several parameters (variance of the underlying conductivity field, measurement time samples and grid discretization). We show that for a certain range of measurement times, the corresponding inverse Laplace transforms can be sped up simultaneously using shared preconditioners across all shifts allowing for additional speedup. 

Additionally, we derive a method for computing the Jacobian using the adjoint state technique based on the Laplace transform. Computing the Jacobian this way requires several solutions of shifted systems of equations. Using the solver developed in our previous work~\cite{saibaba2013flexible} allows for solving the forward (and adjoint) problem for multiple shifts at a cost that is comparable to the cost of solving a system of equations with a single shift. It should be noted that the fast computation of the Jacobian can be used not only to accelerate the geostatistical approach but also other techniques for inversion and data assimilation such as 3D Var, 4D Var filtering and Extended Kalman filter. An additional advantage for using the proposed solver to compute the Jacobian is that the rows corresponding to difference time points can be computed independently without  the need to store the entire time history. We demonstrate results using our method to accelerate an example problem in THT. 

\textbf{Outline}: The paper is organized as follows. We describe the governing equations for groundwater flow and the choice of contour used to approximate the inverse Laplace transform in Section \ref{sec:forward}. In Section~\ref{sec:krylov}, we discuss the flexible Krylov subspace solver for solving the resulting shifted linear systems of equations. In Section \ref{sec:results} we demonstrate the performance and robustness of the proposed solver under various test conditions. We follow this with a description of the geostatistical method for solving inverse problems and show an example taken from THT in Section \ref{sec:inversion}. Finally, in Section~\ref{sec:conc} we conclude with a summary and a discussion on future work. 

\section{Forward Problem}
\label{sec:forward}
\subsection{Transient hydraulic tomography}
In this work we will assume that flow in the aquifer of interest can be modeled as confined (with standard, linear elastic storage) or, if unconfined, can be
treated appropriately using a saturated flow model with the
standard linearized water table approximation~\cite{neuman1972theory}. The equations governing groundwater flow through an aquifer for a given domain $\Omega$ with boundary $\partial \Omega$ is composed of the union of three non-intersecting regions - $\partial\Omega_D$, $\partial\Omega_N$ and $\partial\Omega_w$ referring to Dirichlet, Neumann and linearized water table boundaries respectively. 
\begin{align} \label{eqn:timedomain}
S_s(x) \frac{\partial \phi(x,t)}{\partial t} - \nabla \cdot \left(\kappa(x) \nabla \phi(x,t)\right) & =  \quad  f(x,t), & x &\in \Omega \\ \nonumber
\phi(x,t) & = \quad 0, & x & \in \partial \Omega_D  \\ \nonumber 
\nabla \phi(x,t) \cdot {n} & = \quad 0, & x &\in \partial \Omega_N  \\
\nabla \phi(x,t) \cdot {n} & = \quad -S_y\frac{\partial \phi(x,t)}{\partial t}, & x &\in \partial \Omega_w \nonumber
\end{align}
where $\kappa(x)$  $[LT^{-1}]$ is the hydraulic conductivity, $S_s$  $[L^{-1}]$ is the specific storage, $f(x,t)$  $[T^{-1}]$ is the pumping source, $\phi(x,t)$ is the hydraulic head (pressure) and $S_y$ $[-]$ is the specific yield. In 2D aquifers, the terms $\kappa(x)$ $[L^2 T^{-1}]$ and $S_s(x)$ $[-]$ are known as transmissivity and storativity respectively. In this work, the pumping source is modeled as $f(x,t) = q(t)\delta(x-x_s)$, where $x_s$ corresponds to the location of the pumping source and $q(t)$ to the pumping rate. To derive the weak formulation, we multiply by appropriately chosen test functions $\psi$. Integrating by parts 
\begin{equation}
 (S_s\partial_t\phi, \psi)_{\Omega} +  (S_y \partial_t\phi, \psi)_{\partial\Omega_w}  + A(\phi,\psi)  = q(t)( \delta(x-x_s),\psi)_\Omega
 \end{equation}
where  $(u,v)_\Omega \define \int_\Omega uv d\bx $, $(u,v)_{\partial \Omega_w } \define \int_{\partial\Omega_w} uv d\bx$ and $A(u,v) \define \int_\Omega \kappa(\bx) \nabla u \cdot \nabla v d\bx$. In our work, we have used the standard linear finite element approach but extensions to higher order finite elements are possible. The above weak form is then discretized using finite elements as $\phi_h(t) = \sum_{j=1}^N \hat\phi_j(t)u_j(\bx)$ where $u_j(\bx)$ are the finite basis functions, to obtain the following semi-discrete system of equations
\begin{equation}
 M\partial_t\phi_h  + K\phi_h =  q(t)b  \qquad \phi_{h,0}\{ t=0 \} = \mathcal{P}_h\phi_0
\end{equation}
where the matrices $K$ and $M$ have entries $K_{ij} = A(u_j,u_i)$, $M_{ij} = (S_su_j,u_i)_\Omega + (S_yu_j,u_i)_{\partial\Omega_w}$ and the vector $b$ has entries $b_j = (\delta(x-x_s),u_j)_\Omega$. Furthermore, $\mathcal{P}_h$ is the orthogonal projector with respect to the $(\cdot,\cdot)_\Omega$ inner product. Taking the Laplace transform $\hat{\phi}(\cdot,z) = \int_0^\infty \phi(\cdot,t)e^{-zt} dt$, we have
\begin{equation}
(K+zM)\hat\phi_h(z) = \hat{q}(z)b + M\phi_{h,0}
\end{equation}
The field $\phi_h(t)$ can be recovered by applying the inverse Laplace transform on a carefully chosen contour using the formula
\begin{equation}
\phi_h(t) = \frac{1}{2\pi i}\int_\Gamma e^{zt} (K+zM)^{-1}\left(\hat{q}(z)b + M\phi_{h,0}\right) dz
\end{equation}
where $\hat{q}(z)$ denotes the Laplace transform of $q(t)$ and $\Gamma$ is an appropriately chosen contour. We will now discuss possible choices for $\Gamma$. 

\subsection{Choice of contour}\label{sec:contour} To compute the inverse Laplace transform, we will employ the modified Talbot contour 
\begin{equation} \label{eq:talbot}
\Gamma(\theta): \quad  z(\theta) = \sigma +\mu\left(\theta\cot\theta + \nu i\theta \right) \qquad -\pi \leq \theta \leq \pi 
\end{equation}
where $\theta$ is the contour discretization parameter. This contour was analyzed in \cite{trefethen2006talbot} with the contour parameters $\sigma$, $\mu$ and $\nu$ optimally chosen to improve the convergence of the quadrature scheme. The contour is a simple, closed curve that encloses both the eigenvalues of the generalized eigenvalue problem $Kx = \lambda Mx$ and the singularities of $\hat{q}(z)$.\footnote{In this work, we assume that $\hat{q}(z)$ does not have any singularities on the positive real axis. Otherwise, the  contour has to be adapted to account for these singularities. This is discussed in~\cite{in2011contour}.} In \cite{weideman2006optimizing}, it was shown that optimal convergence as $N_z \rightarrow \infty$ keeping $t$ fixed is achieved with $\sigma$ and $\mu$ proportional to the ratio $N_z/t$.

As in \cite{trefethen2006talbot}, we approximate the integral using the Trapezoidal rule on a uniform grid on $[-\pi,\pi]$ using evenly spaced points $\theta_k$ spaced by $2\pi/N_z$,
\begin{equation}
 \label{eqn:quadrature}
\phi_h(t) = \frac{1}{2\pi i}\int_{-\pi}^\pi e^{z(\theta)t}z'(\theta)F(z(\theta))d\theta  \approx \sum_{k = 1}^{N_z} e^{z(\theta_k) t}z'(\theta_k) F(z(\theta_k))
\end{equation}
where we defined $F(z) = (K+zM)^{-1}\left(M\phi_0 + \hat{q}(z)b \right)$. Further, we also define the weights of the quadrature $w_k = -\frac{i}{2N}e^{z(\theta_k) t}z'(\theta_k)$. Since $(K,M)$ is a Hermitian definite pencil and the contour is symmetric, we need only use half the points, 
\begin{equation} \phi_h(t)  \approx \sum_{k = 1}^{N_z} w_kF(z(\theta_k)) \quad = \quad  2\, \text{Real}\left\{ \sum_{k=1}^{N_z/2}w_k F(z_k) \right\}\label{eqn:half}\end{equation}
\begin{figure}[!ht]
\centering
\includegraphics[scale = 0.35]{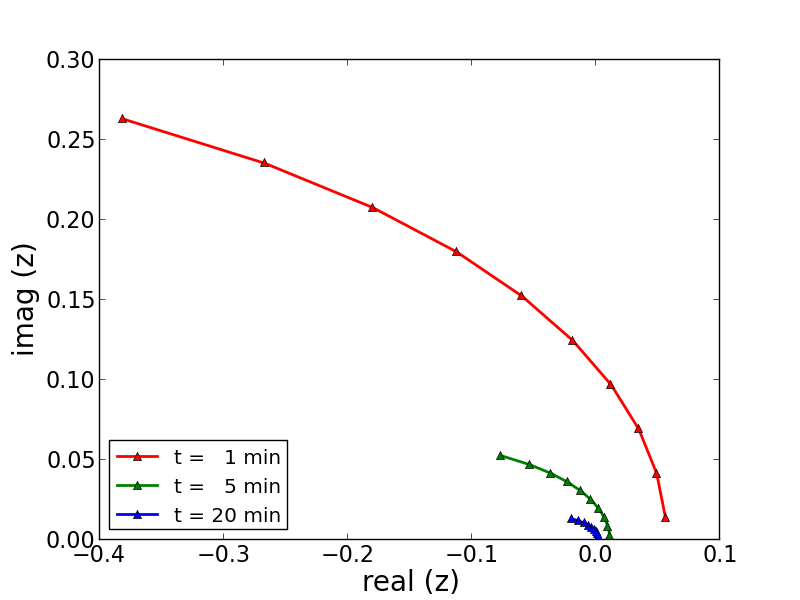}
\includegraphics[scale = 0.35]{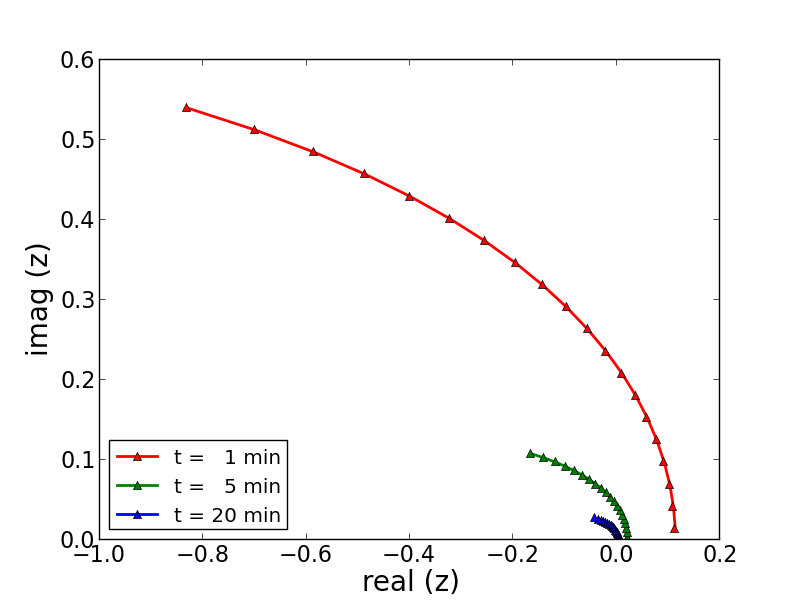}
\caption{Plot of contours corresponding to $N_z = 20$ (left) and $N_z = 40$ (right). Because of symmetry, only half the contour plot is shown. Here $N_z$ is the number of quadrature points in Equation~\eqref{eqn:half}} 
\label{fig:contours}
\end{figure}

An example of the contours for three different measurement times is shown in Figure \ref{fig:contours}. As mentioned earlier, in this paper we only consider sources $f(x,t) = q(t)\delta(x-x_s)$ but this method is applicable for all sources that are separable in space and time. This assumption is crucial since we intend to  use  Krylov subspace methods for shifted systems. This requires the right-hand sides to be independent of the shift (except, perhaps by a multiplicative factor). If the Laplace transform of $q(t)$ is known explicitly, it can be used. Otherwise, one may consider the use of Prony's method to fit an exponential sum $q(t) \approx \sum_j a_j \exp(-\lambda_j t)$ on $[0,T]$, for which the Laplace transform can be computed easily~\cite{sheen2003parallel}. Although in this case, the right hand side is dependent on $z_k$, the Krylov subspace method for shifted systems can still be applied.

Apart from the Talbot contour, other contours such as parabolas and hyperbolas have been proposed and analyzed. In addition, the connection between the Trapezoidal rule applied to approximate the integral and rational approximations to the exponential function have been pointed out in~\cite{trefethen2006talbot}. We would like to emphasize that our fast algorithm does not depend on the specific choice of contour.

\subsection{Accuracy and computational cost}\label{sec:costs}

The error between $\phi_{N_z,h}(t)$ and the true solution $\phi(x,t)$ in the $L_2(\Omega)$ norm can be bounded using the triangle inequality

\[
\|\phi(t) - \phi_{N_z,h}(t)\|\quad  \leq  \quad  \| \phi(t) - \phi_h(t)\| + \|  \phi_h(t) - \phi_{N_z,h}(t)\|
\]
The first term contributing to the error is due to the finite element discretization. As shown in Lemma $3.1$ in \cite{thomee2005high}, this error is 
\[  \| \phi(t) - \phi_h(t)\| \leq Ch^2 \left( \| \phi_0\| + \| \hat{f}\|_\Gamma\right) \]
where, $\| \hat{f}\|_\Gamma \define \sup_{z\in \Gamma} |\hat{f}(z) |$ and $C$ is a constant independent of $h$. The second contributing term is the error incurred by the discretization of the inverse Laplace transform and is dependent on the choice of parameters $\sigma, \mu$ and $\nu$. In summary, the total error is bounded by $\bigO(e^{-cN_z} + h^2)$. For the parameters chosen, $c=\log 3.89$ \cite{trefethen2006talbot}.

We now analyze the computational cost of the Laplace transform method and compare it to a standard time-stepping scheme. Let us assume that we need to compute the solution $\phi(\bx,t)$ at $N_T$ time-steps where $N_T = \bigO(1)$. For a given discretization, let the number of grid points be $N$ and let the cost of solving a linear system of equation be $\mu(N)$. The cost of computing the solution at $N_T$ points in time using the Laplace transform methods is $N_zN_T\mu(N)$ and using time-stepping schemes is $N_t\mu(N)$. Let us now analyze the cost for a desired tolerance $\varepsilon$. Assuming second-order discretization technique in time we have $C\Delta t^2 \sim \varepsilon$, so that $N_t = T/\Delta t \sim C'\varepsilon^{-1/2}$. However, for Laplace transform method the convergence is exponential, so the number of systems is $N_z \sim C'' |\log \varepsilon |$. Thus, for the same accuracy we expect that the Laplace transform-based method is much more efficient. Furthermore, as has been noted in~\cite{weideman2006optimizing,trefethen2006talbot,weideman2007parabolic} and several others, the system of equations can be solved independently for each shift and parallel computing resources can be leveraged for efficient solution of the Laplace transform since the computations are embarrassingly parallel.

Efficient methods have been proposed for solving the sequence of shifted systems arising from the discretization of the contour integral using the Trapezoidal rule. In~\cite{in2011contour}, the authors propose a Krylov subspace method for solving the shifted system of equations. However, there is a rich literature concerning the Krylov subspace methods for shifted systems (for a good review see~\cite{simoncini2007recent}). In practice, efficient preconditioners are needed to ensure convergence in a reasonable number of iterations. For this purpose, in this paper, we adapt the flexible Krylov subspace methods proposed in~\cite{gu2007flexible} and analyzed in~\cite{saibaba2013flexible}.

\section{Fast solvers for shifted systems}\label{sec:krylov}
Recall the approximation of the inverse Laplace transform as shown in Equation \eqref{eqn:quadrature}. It leads to a sequence of shifted systems at each time step, which can then each be solved independently. In particular, for each time step, we need to solve two shifted systems of equations, 

\begin{equation}
\left( K + z_k M \right) X_k = [b,\quad M \phi_0], \qquad k = 1, \dots, N_z/2 
\end{equation}
where $z_k$ are the (complex) shifts corresponding to the inverse Laplace transform contour. $F(z_k)$ is thus found by adding the solution to the first system of shifted systems to the second (times the corresponding $\hat{q}(z_k)$). In this section, we consider Krylov subspace methods for the solution of the first shifted systems of equations,                     
\begin{equation}
 \label{eqn:multipleshifted}
  \left( K + z_k M \right) x_k  = b, \qquad k =1,\dots, N_z/2
\end{equation}
It should be noted that the second system of equations can be solved for in a similar fashion. We assume that the systems are non-singular for all $k=1,\dots,N_z/2$. Recall that because of Equation~\eqref{eqn:half} only half the systems need to be solved. In Equation~\eqref{eqn:multipleshifted}, the matrices $K$ and $M$ and right-hand side $b$ are independent of the shifts $z_k$. Krylov based methods are particularly attractive for this class of problems because of their shift-invariant property. The main strategy of Krylov solvers is to build a Krylov basis and then to search for an approximate solution in the resulting reduced space.  The expensive step of constructing the Krylov basis needs to be performed only once since an approximation space is generated independently of the shifts. Once the basis is built, the subproblem for each shift can then be solved at a much reduced cost. We consider Arnoldi-based Krylov solvers, in particular the Full Orthogonalization Method (FOM) and the General Minimum RESidual Method (GMRES) \cite{saad2003iterative}. 

The number of iterations required by the Krylov subspace solvers for convergence can be quite large, particularly for problems arising from realistic applications. To reduce the number of iterations, we use right preconditioners of the form $K_{\tau} \define (K+\tau M)$ factorized and inverted using a direct solver. For preconditioners of this form and for $k = 1 \dots N_z/2$ it can be readily verified that, 

\begin{equation}\label{eqn:shiftinvariant}
 (K + z_k M) (K + \tau M)^{-1} = I + (z_k - \tau) M (K+\tau M)^{-1}
\end{equation}
It can be shown that $\krylov{MK_\tau^{-1},b}{m}$ is equivalent to the Krylov subspace $\krylov{(K+z_k M)(K+\tau M)^{-1},b}{m}$. This property is known as the shift-invariant property. The algorithm to solve multiple shifted systems using a single shift-and-invert preconditioner is as follows: a basis for the Krylov subspace $\krylov{MK_\tau^{-1},b}{m}$ is obtained by running $m$ steps of the Arnoldi algorithm on the matrix $M K_\tau^{-1}$ with the starting vector $b$. At the end of $m$ steps, the following Arnoldi relationship holds,
\begin{equation}
 \label{eqn:shiftedarnoldi}
 \left(K  + z_k M \right)K_\tau^{-1}V_m = V_{m+1} \underbrace{\left(\begin{bmatrix} I \\ {0}\end{bmatrix} + (z_k-\tau) \bar{H}_m  \right)}_{\define \bar{H}_m(z_k;\tau)} = V_{m+1}\bar{H}_m(z_k;\tau)
 \end{equation}
where $V_{m+1} = [v_1,\dots,v_{m+1}]$ is the Krylov basis and $\bar{H}_m$ is a $(m+1) \times m $ matrix formed by the Arnoldi process. We use the notation $H_m$ to denote the corresponding upper Hessenberg matrix, i.e. $H_m  = [ I ,\;  0 ] \bar{H}_m$. Having constructed the Arnoldi basis, the solution for each shift is obtained by searching for solutions $x_m(z_k) \in  \Span{K_\tau^{-1}V_m}$, written as $x(z_k) = K_\tau^{-1}V_my_m(z_k)$ where the vectors $y_m(z_k)$ can be obtained by either by imposing a Petrov-Galerkin condition, i.e. $r_m(z_k) \perp V_m$ which leads to the Full Orthogonalization Method for shifted systems (FOM-Sh) or by minimizing the residual norm $\normtwo{r_m(z_k)}$ over all possible vectors in the span of $K_\tau^{-1}V_m$ which leads to the GMRES method for shifted systems (GMRES-Sh). The residual is defined as 
\[r_m(z_k) \define b - (K+z_kM)x_m = V_{m+1}\left( \beta e_1 -\bar{H}_m(z_k;\tau) y_m(z_k) \right)\] 

Numerical evidence in~\cite{saibaba2013flexible} showed that a single preconditioner might not be sufficient to effectively precondition all systems, especially if the values of $z_k$ are spread out relative to the spectrum of the matrices. This is because a single preconditioner only adequately preconditions systems with $z_k$ close to $\tau$. In order to improve the convergence rate across all the shifts, we proposed a Flexible GMRES/FOM solver~\cite{saibaba2013flexible} for shifted systems that allows us the flexibility to choose different preconditioners (of the shift-and-invert type) at each iteration. This requires storing an additional set of vectors $U_m = [u_1,\dots,u_m]$. Consider the use of a preconditioner at each iteration of the form $K + \tau_j M$ where $j = 1, \dots, m$ and $j$ is is the iteration number. Systems with shifts $z_k$ close to $\tau_j$ converge faster. The procedure to solve the shifted systems of equations follows similarly to the single preconditioner case with the following modified Arnoldi relationship~\cite{saibaba2013flexible},

\begin{equation}
(K + z_k M)U_m = V_{m+1} \underbrace{\left( \begin{bmatrix} I \\ 0 \end{bmatrix}  + \bar{H}_m (z_kI_m - T_m) \right)}_{\define \bar{H}(z_k; T_m)} = V_{m+1}\bar{H}_m(z_k;T_m) \label{eqn:shiftedarnoldimod} 
\end{equation}
where $T_m = \text{diag}\{ \tau_1,\dots,\tau_m\}$. The columns of the matrix $U_m$ now no longer span a Krylov subspace but rather a rational Krylov subspace. As in the case of a single preconditioner $K_\tau^{-1}$, the approximate solution is constructed $x_m(z_k) = U_my_m(z_k)$, so that $x_m(z_k) \in \Span{U_m}$. Again, the coefficients $y_m(z_k)$ are either computed by using a Petrov-Galerkin condition, leading to the flexible FOM for shifted systems (FFOM-Sh) or by minimizing the residual, leading to flexible GMRES for shifted systems (FGMRES-Sh). The expensive step of computing the matrix $U_m$ is shared across all the systems and the cost of solving smaller subproblem (either by a Petrov-Galerkin projection or a residual minimization) for each shift is negligible compared to the cost of generating $U_m$. The algorithm is presented in detail in Algorithm~\ref{alg:flexible}.

\begin{algorithm}[!ht]
\begin{algorithmic}[1]

  \STATE Given $K$, $M$, $b$, \{$\tau_1, \dots, ,\tau_m$\}, $ \{z_1, \dots, z_{N_z/2}\} $ and a tolerance $tol$,
  
  \STATE $v_1 = b /\beta$, $k = 1$ and $\beta \define \normtwo{b}$
  \FORALL {$j = 1,\dots, \text{maxit}$}
  \STATE Solve $ (K+\tau_j M) u_j =  v_j$
  \STATE $s_j := Mu_j$
   \FORALL {$i=1,\dots,j$} 
    \STATE $h_{ij} := s_j^*v_i$
    \STATE Compute $s_j := s_j - h_{ij}v_j$
  \ENDFOR 
  \STATE $h_{j+1,j}:= \lVert s_j \rVert_2 $. If $h_{j+1,j} = 0$ stop 
  \STATE $v_{j+1} = s_j/ h_{j+1,j}$
   \STATE Define $U_j = [u_1,\dots,u_j] $, $V_{j+1} = [v_1,\dots,v_{j+1}]$ and $\bar{H_j} = \{h_{i,l} \}_{1\leq i \leq j+1, 1 \leq l \leq  j}$
   \STATE Construct $T_j = \text{diag}\{\tau_1,\dots,\tau_j\}$ 
  \FORALL {$k = 1,\dots,N_z/2$}
   \STATE If, not converged 
   \STATE Construct $\bar{H_j}(z_k;T_j)= \begin{bmatrix} I \\ 0 \end{bmatrix}  + \bar{H_j} (z_k I_j - T_j) $
   \STATE FOM: 
    \[  H_j(z_k; T_j)y_j^{fom} =  \beta e_1 , \quad \text{where } H_j(z_k; T_j) = [I, 0] \bar{H_j}(z_k; T_j)\]
   \STATE GMRES: 
    \[ y_j^{gm} \define \min_{ y_j \in \mathbb{C}^j} \normtwo{ \beta e_1 - \bar{H}_j(z_k;T_j) y_j }\] 
    \STATE Construct the approximate solution as $x_j(z_k) = U_jy_j(z_k)$
  \ENDFOR
\STATE If all systems have converged, exit
 \ENDFOR
\end{algorithmic}
\caption{Flexible FOM/GMRES for shifted systems}
\label{alg:flexible}
\end{algorithm}

In Algorithm~\ref{alg:flexible}, the use of a different preconditioner shift at each iteration would be too expensive since a different preconditioner $(K+\tau_jM)$ has to be factorized at each iteration. Numerical evidence shows that this is unnecessary since we need only pick a few systems with shifts $\tau$ such that they effectively precondition the systems over the entire range of shifts. The rule of thumb is that systems with shifts $z_k$ close to $\tau$ converge first. In our previous work \cite{saibaba2013flexible}, we considered the solution of shifted systems with $200$ shifts that were on the pure imaginary axis.  We observed that the systems with shifts closer to the origin converged more slowly. Therefore, to speed convergence, we chose $5$ preconditioners with shifts $\tau$ that were evenly spaced on a log scale. The choice of shifts and number of preconditioners is entirely application dependent. For the case of the shifts determined by the modified Talbot contours as shown in Figure \ref{fig:contours}, numerical experiments in Section~\ref{sec:results} show the use of two preconditioners produces more favorable results compared to only a single preconditioner. The shifts are chosen as follows: the first shift corresponds to $z_k$ with the smallest real part called $\bar\tau_1$, and the second shift corresponds to the shift with the largest real part called $\bar\tau_2$. The sequence $\tau_j$ for $j=1,\dots,m$ is constructed as follows. We choose $\tau_j = \bar\tau_1$ for $j=1,2,3$ and $\tau_j = \bar\tau_2$ for $j=4,5$. This procedure is repeated until all the systems converge to the desired user-defined tolerance. This choice of preconditioners results in fast convergence for the range of parameters we have explored. Future work will focus on the optimal choice of preconditioners.

\subsection{Properties of approximation space}
In this section, we analyse the convergence of the flexible Krylov subspace approximation to the discretized representation of the contour integral. The error of the Krylov subspace approach in Equation~\eqref{eqn:shiftedarnoldimod} is defined by~\cite{popolizio2008acceleration,lopez2006analysis} as 
\begin{equation}
e_m \define \sum_{k=1}^{N_z} w_k \left( (K + z_kM)^{-1}b  - x_m\right) = \sum_{k=1}^{N_z} w_k \left( (K + z_kM)^{-1}b - U_m H_m^{-1}(z_k;T_m)\beta e_1 \right) 
\end{equation}
 where $x_m = U_m H_m^{-1}(z_k;T_m)\beta e_1$ is the approximate projected solution. Similarly, we can define the residual vector 
\begin{equation} r_m \define = sum_{k=1}^{N_z} w_k \left( b - (K+z_k M)x_m\right) =  \sum_{k=1}^{N_z} w_k \left( b - (K+z_k M)U_mH_m^{-1}(z_k;T_m)\beta e_1 \right) \end{equation}
It can be readily verified that the approximate solution $x_m \define \sum_k w_k U_m H_m^{-1}(z_k;T_m)\beta e_1$ lies in the subspace $\Span{U_m}$ where 
\begin{equation}\label{eqn:rationalkrylov}
\Span{U_m} = \Span{ v_1, T_1v_1,\dots,(T_{m-1}\cdots T_1)v_1 }, \text{
with } T_k = (K+\tau_k M)^{-1}
\end{equation}
To simplify the analysis, let us make the following change of variables: $b \leftarrow M^{-1/2}b$, $A \leftarrow M^{-1/2}KM^{-1/2}$ and $x \leftarrow M^{1/2}x$. This variable change is possible since $M$ is a mass matrix and therefore, positive definite.

The projective space $U_m$ is independent of the shift $z_k$ and is not a typical Krylov subspace but rather a rational Krylov subspace.  Krylov subspaces generate sequences of vectors that are of polynomial form, i.e., $v_{m+1} = p_m(A) v_1$ where $p(\cdot)$ is a polynomial. On the other hand, rational Krylov subspace generates iterates that can be expressed as rational functions of the matrix, i.e. $v_{m+1} = y_{m-1}(A)z_m(A)^{-1}v_1$, where $y_{m-1}$ is a polynomial of degree at most $m-1$ and $z_m$ is a fixed polynomial of degree $m$. The rational Krylov method was originally proposed by
Ruhe~\cite{ruhe1984rational} in the context of approximating interior eigenvalues, with which appropriately chosen shifts could potentially accelerate convergence of the eigenvalues in the desired region of the spectrum. They have been quite successful in the context of model reduction, in order to approximate the transfer function for a dynamical system~\cite{druskin2011adaptive}. The rational approach requires either knowledge of the shifts $\tau_k$ a priori, or this needs to be computed on-the-fly. In this work, our choice of shifts is done a priori and is based on heuristics that perform well in the applications, as we will demonstrate. In the context of $\mathcal{H}_2$-optimality reduction, an automated choice of shifts has been proposed in~\cite{druskin2011adaptive}. The following result characterizes the properties of the error due to the flexible Krylov approach and derives an expression for the rational function that defines the rational Krylov subspace. 

\begin{propos}
The error $e_m$ satisfies
\begin{equation} e_m \in \Span{(A+z_{1}I)^{-1}v_{m+1}, \dots,(A+z_{N_z}I)^{-1}v_{m+1}} \label{eqn:span}\
\end{equation}
and can be expressed in terms of a rational function as 
\begin{equation} e_m = r(A)v_1 \qquad 
r(\lambda) \define \sum_{k=1}^{N_z}\frac{w_k}{\lambda + z_k}\frac{\text{det}(G_m - \lambda H_m)}{\prod_{j=1}^m(\lambda+\tau_j)} \label{eqn:rational}\end{equation}
where $G_m = I - H_mT_m$ Further, the error can be bounded as 
\begin{equation} \normtwo{e_m} \quad \leq   \quad \sum_{k=1}^{N_z/2}  |w_k| \normtwo{(A+z_kI)^{-1}v_{m+1}} |\eta_k|
\end{equation}
where $\eta_k \define h_{m+1,m} (z_k - \tau_m) e_m^*H_m^{-1}(z_k;T_m)\beta e_1$.
\end{propos}

\begin{proof}

Let us begin by writing 
\[ e_m = \sum_{k=1}^{N_z} w_k (A+z_kI)^{-1} \left(b - (A+z_kI)U_m H_m^{-1}(z_k;T_m)\beta e_1 \right) \]
Using the modified Arnoldi relation in equation~\eqref{eqn:shiftedarnoldimod}
\begin{equation} b - (A+z_kI)U_m H_m^{-1}(z_k;T_m)\beta e_1 = b - V_{m+1}\bar{H}_m(z_k;T_m)H_m(z_k;T_m)^{-1}\beta e_1 = - v_{m+1}\eta_k \end{equation}
since $b = V_m \beta e_1$ and $\eta_k$ as defined above. Therefore, we have 
\[ e_m = - \sum_{k=1}^{N_z} w_k \eta_k (A+z_kI)^{-1}v_{m+1}  \]
This proves the part of the proposition in Equation~\eqref{eqn:span} the Arnoldi relations $U_m = V_{m+1}\bar{H}_m$ and $AU_m + U_m T_m = V_m$. Eliminating $U_m$ we have $AV_{m+1}\bar{H}_m + V_{m+1}\bar{H}_mT_m = V_m$. 
The general recurrence for the vector $v_m$ can be written as 
\[ ( A + \tau_m I)v_{m+1} h_{m+1,m} = v_j - \sum_{j=1}^m \left(Av_jh_{jm} + v_jh_{jm}\tau_m \right)\]
The individual vectors $(A+z_kI)^{-1}v_{m+1}$ can be expressed in terms of rational functions and n convergence this yields $h_{m+1,m} = 0$ following the arguments of~\cite{ruhe1984rational} we get that 
\[ v_{m+1} = r(A)v_1 \quad \text{where} \quad r(\lambda) \define \frac{I - H_m T_m - \lambda H_m}{(\lambda+\tau_1)\cdots(\lambda+\tau_m)}  \]
Plugging in this expression into Equation~\eqref{eqn:rational} we get the desired result. The inequality for the error follows from the properties of vector norms.

\end{proof}
We also have that the residual satisfies $r_m = - \sum_{k=1}^{N_z} w_k \eta_k  v_{m+1} $ and therefore, $\normtwo{r_m} \leq \sum_k|w_k| |\eta_k|$. The quantity $|\eta_k| = |h_{m+1,m} (z_k - \tau_m) e_m^*H_m^{-1}(z_k;T_m)\beta e_1|$ is established as an a posteriori estimate of error of the flexible Krylov approach~\cite{saibaba2013flexible}. The result in Equation~\eqref{eqn:rational} suggests that an efficient choice of shifts $\tau_k$ for the preconditioners may result in a smaller error $e_m$ after $m$ iterations. In~\cite{saibaba2013flexible} analysis is provided for when the a posteriori measure is small and showed that it is related to the convergence of the eigenvalues of the matrix $KM^{-1}$ (or alternatively, $M^{-1/2}KM^{-1/2}$). Although we do not have an automated procedure to determine the shifts, empirical evidence supports that our heuristic choice of preconditioner shifts $\tau_k$ leads to fast convergence for a wide range of parameters of the underlying PDE. This is discussed in the following section. 

\section{Numerical Experiments}
\label{sec:results}
\subsection{Problem set-up}
In this section we demonstrate using numerical examples the computational gains of the flexible Krylov solver and its robustness to changing various parameters in the forward problem. We consider a 2D depth-averaged aquifer with horizontal confining layers in a square domain satisfying the differential equation~\eqref{eqn:timedomain} with zero Dirichlet boundary conditions on all boundaries. The equations are discretized using standard linear finite elements implemented using FEniCS ~\cite{LoggMardalEtAl2012a, LoggWells2010a, LoggWellsEtAl2012a} using Python as the user-interface. The domain size is $[0,L]^2$ where $L= 100 $ [m].

For our numerical experiments, we consider two transmissivity fields : the first is a random field generated using an exponential covariance kernel, i.e. \begin{equation}\kappa(x,y) = \theta \exp( -r) \label{eqn:exp} \end{equation} with $r=\normtwo{x-y}/L$. For the second transmissivity field we use a scaled version of Franke's function~\cite{franke1979critical} which is a smooth function obtained by the weighted sum of four exponentials. The natural log of the fields are displayed in Figure~\ref{examplefields}.  The exponential field is rougher than the Franke field and as a result we expect the iterative solver to work hard to converge to the desired tolerance. In both cases, the mean transmissivity was chosen to be $\mu_{K} = 10^{-4}$ [m$^2$/s] and in these examples storativity was chosen to be constant with $S_s = 10^{-5}$ [-]. We consider one pumping source located at $(50,50)$ pumping at a constant rate of $0.85$ L/s.

 We compare the results of two different solvers - `Single' corresponds to the Krylov solver with a single preconditioner $(K+\tau M)$ where, $\tau = \arg\min \text{Real}(\sigma_k)$  for $k = 1,\dots, N_z$, i.e. the shift corresponding to the smallest real part and `Flexible' using the flexible Krylov approach described in section~\ref{sec:krylov}. The procedure to choose the shifts for the preconditioners is also described there.

\begin{figure}
\centering
\includegraphics[scale = 0.35]{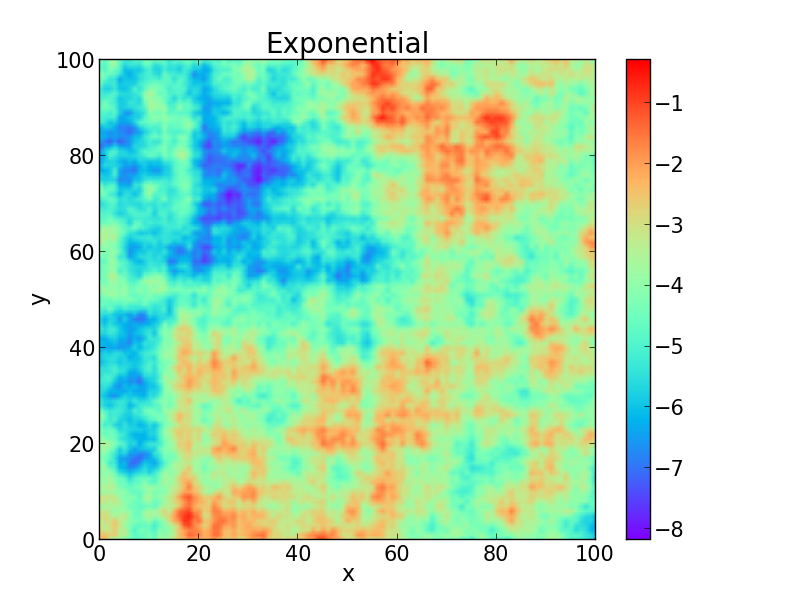}
\includegraphics[scale = 0.35]{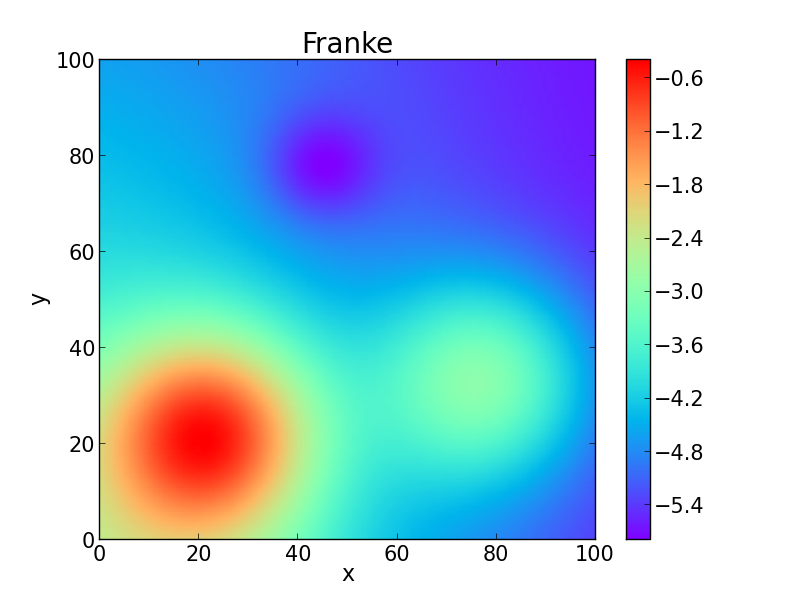}

\caption{Example log transmissivity fields generated by (left) sample drawn from a random field with exponential covariance kernel  and (right) scaled version of Franke's function. The variance for both fields in the figure is $1.6$.} \label{examplefields}
\end{figure}

\subsection{Results}

To demonstrate the robustness of our solver, we test it with respect to different parameters - variance of the field, number of times at which the solution needs to be computed and the number of grid points in the domain. First, we vary the variance of the log transmissivity field and note that higher variances correspond to more ill-conditioned systems. We report the number of iterations taken and the CPU time required for the solvers to converge to a relative tolerance of $10^{-10}$. The number of shifts $N_z$ was chosen to be $40$ and therefore the number of systems required to be solved is $N_z/2 = 20$. The results are displayed in Table~\ref{table:variance}. The last column in the table is an estimate of the maximum condition number of $K+\sigma M$ across all shifts. This is a lower bound to the condition number in the $1$-norm \cite{hager1984condition, higham2000block}. We see that the fields corresponding to larger variances have more ill-conditioned systems. Since Franke field is smoother than the random field, it is expected that the resulting linear systems are less ill-conditioned and therefore, it would take fewer number of iterations to converge to the desired tolerance. Furthermore, the number of iterations required by both `Single' and `Flexible' preconditioning increases with increasing variance, however the number of iterations using `Flexible' preconditioning does not increase significantly. By contrast, the use of `Single' preconditioning is insufficient because it does not adequately precondition all the systems and the cost can grow significantly.

\begin{figure}
\centering
\includegraphics[scale = 0.4]{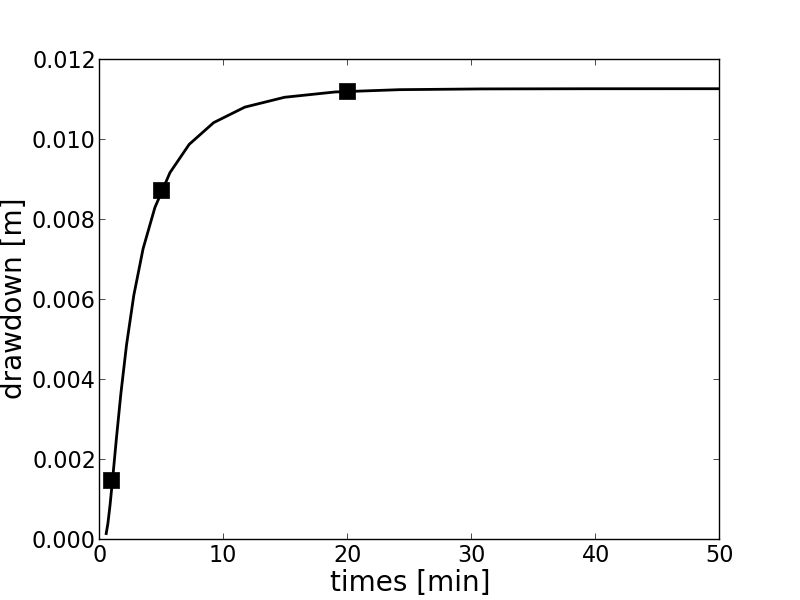}
\caption{Drawdown curve measured at location $(70,70)$ with the source at $(50,50)$ pumping at a constant rate of $0.85$ L/s. The squares indicate sampling times $1$, $5$ and $20$ min that were used in the results in Table \ref{table:times}.} \label{fig:drawdown}
\end{figure}

\begin{table}[!ht]
\begin{center}
\begin{tabular}{|c|c|c|c|c|l|}\hline
\multicolumn{6}{|c|}{Franke Field} \\ \hline
\multirow{2}{*}{Variance} & \multicolumn{2}{|c|}{Single} & \multicolumn{2}{|c|}{Flexible} & \multirow{2}{*}{$\kappa_{\text{pseudo}}$ }\\  \cline{2-5}
& Iter. & CPU Time [s] & Iter. & CPU Time [s] & \\ \hline
0.8 & 73  & 8.6 & 40 & 4.6  & $ 1.2 \times 10^{6}$\\ \hline 
1.6 & 95  & 13.0 & 48 & 5.8  & $ 1.5 \times 10^{7}$\\  \hline
3.5 & 125 & 21.3 & 54 & 7.1  & $ 1.3 \times 10^{9}$\\ \hline 
\multicolumn{6}{|c|}{Random Field} \\ \hline
& Iter. & CPU Time [s] & Iter. & CPU Time [s] & \\ \hline
0.8 & 89  & 11.6  & 45 & 5.5  & $ 1.1 \times 10^{6} $\\ \hline 
1.6 & 126  & 20.1 & 53 & 6.8  & $ 9.9 \times 10^{6} $\\  \hline
3.5 & 139 & 28.9 & 54 & 7.3  & $ 9.8 \times 10^{8} $\\ \hline

\end{tabular}
\caption{Iteration count and CPU time  for solving the system of equations~\eqref{eqn:multipleshifted}. The solution was computed at measurement time = $5$ min. Two different log transmissivity fields were considered. An estimate for condition number shows that the fields with higher variance corresponded to larger condition number. The proposed `Flexible' solver performs better than `Single', both in terms of iterations and CPU time. } \label{table:variance}
\end{center}
\end{table}

Next, we consider the performance of the solvers in computing the solution at three different measurement times. The number of iterations and the time required by the solver for solving for the hydraulic head at various times is computed and displayed in Table~\ref{table:times}. It is observed that the number of iterations required for the systems corresponding to the early times to converge are larger, even though the condition number decreases. The reason for the increase in iterations can be explained by examining the the contours obtained from  different measurement times (Figure \ref{fig:contours}). It is the presence of the scaling factor $1/t$ multiplying the shifts $z_k$ that cluster the shifts for larger times, see Equation \eqref{eq:talbot}. The clustering of the shifts allows for faster convergence because it allows for better preconditioning (both `Single' and `Flexible') of all the systems since the systems with shifts closer to $\tau$ will converge more rapidly than those further away. This is consistent with the analysis done in \cite{saibaba2013flexible}.

\begin{table}[!ht]
\begin{center}
\begin{tabular}{|c|c|c|c|c|l|}\hline
\multicolumn{6}{|c|}{Franke Field} \\ \hline
\multirow{2}{*}{Meas. Time [min]} & \multicolumn{2}{|c|}{Single} & \multicolumn{2}{|c|}{Flexible} & \multirow{2}{*}{$\kappa_{\text{pseudo}}$ }\\  \cline{2-5}
& Iter. & CPU Time [s] & Iter. & CPU Time [s] & \\ \hline
1 & 134  & 22.6 & 55 & 7.8  & $ 4.3 \times 10^{6} $\\ \hline 
5 & 95  & 13.0 & 48 & 5.8  & $ 1.5 \times 10^{7} $\\  \hline
20 & 58 & 5.8 & 33 & 3.4  & $ 5.6 \times 10^{7} $\\ \hline 
\multicolumn{6}{|c|}{Random Field} \\ \hline
& Iter. & CPU Time [s] & Iter. & CPU Time [s] & \\ \hline
1 & 145  & 26.2  & 59 & 8.3 & $ 3.0 \times 10^{6} $\\ \hline 
5 & 126  & 20.1 & 53 & 6.8  & $ 9.9 \times 10^{6} $\\  \hline
20 & 92 & 12 & 45 & 5.4  & $ 3.8 \times 10^{7} $\\ \hline

\end{tabular}
\caption{Iteration count and time for for solving the system of equations~\eqref{eqn:multipleshifted}. Two different log transmissivity fields were considered and the solution was computed at different measurement times. Here, variance of the log transmissivity was fixed to be $1.6$. An estimate for condition number shows that the systems at earlier measurement time  corresponded to larger condition number. The proposed `Flexible' solver performs better than `Single', both in terms of iterations and CPU time.} \label{table:times}
\end{center}
\end{table}
Thus far, we have only looked at solving for each time point independently, however our solver allows for simultaneously solving for the solutions at multiple times. We consider $40$ uniformly spaced time points between $40$ min and $60$ min. These measurements correspond to `late-time' measurements (see Figure~\ref{fig:drawdown}). These systems are solved simultaneously using `Flexible' preconditioning. To pick the preconditioner shifts, we first sort all the shifts arising from the discretized contour from the $40$ time instances by their real parts and we choose two preconditioners $\tau_1$ corresponding to the overall shift with the largest real part and $\tau_2$ corresponding to the overall shift with the smallest real part. We observe that the maximum number of iterations was $27$ across all shifts and all the $40$ time sample points, which  is roughly the same cost of solving for one single time, i.e., the computational cost associated with solving the problem for multiple times is not significantly higher than the cost of solving the system of equations for a single time. More investigation is needed to extend the solvers to be able to compute the entire time history simultaneously.

\begin{table}[!ht]
\begin{center}
\begin{tabular}{|c|c|c|c|c|}\hline
\multirow{2}{*}{Grid size} & \multicolumn{2}{|c|}{Single} & \multicolumn{2}{|c|}{Flexible} \\ \cline{2-5}
& Iter. & Time [s] & Iter. & Time [s] \\ \hline
$41^2$ & 96  & 3.6 & 49 & 1.5  \\ \hline 
$101^2$ & 95  & 13.0 & 48 & 5.8  \\  \hline
$201^2$ & 94 & 52.8 & 45 & 25.1  \\ \hline 
$301^2$ & 92 & 260.2 & 44 &  116.2 \\ \hline
$401^2$ & 91 & 611.7 & 44 & 268.1 \\ \hline

\end{tabular}
\caption{Iteration count and time respectively for computing pressure field using increasing number of grid points. The log transmissivity field is Franke field with a variance of $1.6$.} \label{table:grid}
\end{center}
\end{table}

All of the above experiments have been conducted on a grid of size $101^2$. To show the independence of the iteration count with respect to the grid size, we consider grids ranging from $41^2$ to $401^2$. In all the experiments, the log transmissivity field is the Franke field with a variance of $1.6$. Table~\ref{table:grid} lists the iteration count of the solver with increasing grid size. For both the solvers `Single' and 'Flexible', the iteration count is independent of the number of grid points. 
 
In summary, we can conclude that our solver is robust to various parameters such as field smoothness and variance, as well as the computation of the solutions at different measured times. Additionally, for `late enough' times, the proposed flexible solver further improves efficiency as it allows for the solution of the resulting systems of equations simultaneously. Therefore, this solver can be used within the framework of inverse problems to achieve significant speedup.

\section{Inverse Problem}
\label{sec:inversion}

\subsection{Geostatistical approach}

We consider the geostatistical approach as described in \cite{kitanidis1995quasi, kitanidis2010bayesian}, which is one of the prevalent approaches for solving the inverse problem. The objective is to determine aquifer properties (here log transmissivity) given discrete head measurements. In the geostatistical approach the unknowns are modeled as Gaussian random fields and the Bayesian approach is used to infer the posterior probability density function of the unknowns as the product of two parts: the likelihood of the measured data or the `data misfit' and the prior distribution of the parameters which represents the assumed spatial correlation of the parameters. Denote by $s(\bx) \in \mathbb{R}^{N_s}$ the vector corresponding to the discretization of the function to be estimated with $s \sim {\cal{N}}(X\beta,Q)$ where $X$ is a $N_s \times p$ known matrix, $\beta$ are $p$ unknown drift coefficients and $Q$ is a covariance matrix with entries $Q_{ij} = \kappa(x_i,x_j)$ where $\kappa$ is a generalized covariance kernel. The covariance kernel contains information about the degree of smoothness of the random field and the correlation between two points. Popular choices of covariance kernels include the Mat\'ern covariance family, which includes the exponential and Gaussian covariance kernels~\cite{rasmussen2006gaussian}. The measurement equation is expressed as, 

\begin{equation}
y = h(s) + \epsilon \quad \epsilon \sim {\cal{N}}(0,R)
\end{equation}
where $y \in \mathbb{R}^{N_y}$ denotes the hydraulic head measurements and $\epsilon$ represents the measurement error which includes both the error resulting from data collection and the errors in evaluating $h(s)$. The matrices $Q,R$ and $X$ are structural parameters whose values can be optimized using a restricted maximum likelihood approach. More details on the choice of these parameters can be found in \cite{kitanidis1995quasi}. The best estimate is obtained by computing the maximum a posteriori estimate (MAP). This is equivalent to solving the following optimization problem,
\[ \underset{\hat{s},\hat{\beta}} {\arg \min} \quad \frac{1}{2} \norm{y-h(s)}{R^{-1}}^2 + \frac{1}{2}\norm{s - X\beta}{Q^{-1}}^2\]
where, the objective function to be minimized is the negative logarithm of the posterior probability density function $p(s,\beta|y)$. To solve the nonlinear optimization problem, Gauss-Newton algorithm is used. The procedure is described in Algorithm~\ref{alg:quasi}.

\begin{algorithm}[!ht]
\begin{algorithmic}
\WHILE {not converged}
\STATE  Compute the $N_y \times N_s$ Jacobian $J$ as, 
\begin{equation} 
 J_k = \at{\frac{\partial{h}}{\partial{s}}} {s = {s}_k} 
\end{equation}
\STATE  Solve the following system of equations, 

\begin{equation} \label{eq:inversion1}
\left( \begin{array}{cc} 
        J_k Q J_k^T +R  & J_kX \\
	 \left(J_k X\right)^T & 0 
       \end{array}
       \right) 
       \left( \begin{array}{c} \xi_{k+1} \\ \beta_{k+1} \end{array} \right) = 
\left( \begin{array}{c} y - h({s}_k) + J_k s_k\\ 0 \end{array} \right) 
\end{equation}

\STATE $s_{k+1}$ is computed by, 
\begin{equation} 
 s_{k+1} = X \beta_{k+1} + Q J_k^T \xi_{k+1} 
\end{equation}

\ENDWHILE
\end{algorithmic}
\caption{Quasi-linear Geostatistical approach}
\label{alg:quasi}
\end{algorithm}

Every iteration in algorithm~\ref{alg:quasi} requires the computation matrices $J_kQJ_k^T$ and $Q J^T_k$. Storing the dense covariance matrix $Q$ while computing $Q J_k^T$ can be expensive, both in terms of memory costs for storage and computational costs, particularly when the number of grid points is large. For covariance matrices that are translation invariant or stationary, the associated cost for storing $Q$ and computing matrix vector products $QJ_k^T$ can be reduced. A Fast Fourier Transform (FFT) method can be used if the grid is regular. If the grid is not regular, the Hierarchical matrix approach can be used \cite{saibaba2012efficient}. The resulting computations are $\bigO(N_yN_s\log N_s)$ instead of $\bigO(N_yN_s^2)$.

\subsection{Sensitivity computation}

Another step that can be very expensive in the geostatistical method for inversion is the computation of the Jacobian, or sensitivity matrix. When traditional time stepping algorithms are used for simulating the time-dependent forward problem, computing the sensitivity of the measurement operator with respect to parameters to be inverted for is accomplished by the adjoint state method. This involves cross-correlation of two fields at the same time, one obtained by forward recursion and the other by backward recursion. Typically, the forward recursion is performed first, however the entire time history up to the desired time must be accessible during the backward recursion. For small scale problems, one typically stores the entire time history. However, for large-scale problems arising from finely discretized problems the storage requirements may be so large that it may not be able to be stored in RAM so one must resort to disk storage, in which case memory access then becomes the limiting factor. A standard approach is to use checkpointing~\cite{griewank2000algorithm,symes2007reverse,wang2009minimal} which trades computational complexity (by a factor logarithmic in number of steps) while reducing the memory complexity to a few state buffers (also logarithmic in number of steps).

Using the Laplace transform approach to compute the solution of the forward problem at time $t$ avoids the expensive computation of the time history. In this section, we derive expressions for computing the sensitivity of the measurements with respect to the parameters also using the Laplace transform. In addition to the solution of the forward problem $\phi(\bx,t)$, the algorithm requires computing an adjoint field $\psi(\bx,t)$ solving for which also leads to a shifted system of equations, and can be efficiently done using the flexible Krylov subspace method described in section~\ref{sec:krylov}.

We now derive expressions for the sensitivity. Let $S_s$ be parametrized by $s_S$ and $\kappa$ be parametrized by $s_K$. Since we want the reconstruction of $S_s$ and $\kappa$ to be positive, it is common to consider a log-transformation $S_s = e^{s_S}$ and $\kappa = e^{s_K}$. Consider the functional that we would like to compute the sensitivity of 
\[ I(t;s_K, s_S) \define \int_\Omega \delta(\bx-\bx_m) \phi(\bx,t) d\bx \]
which corresponds to a point measurement at the measurement location. We are interested in computing the sensitivity of the functional $I$ with respect to the parameters $ s_K$ and $s_S$. We start by making the following transformation 
\[ I(t;s_K, s_S)  = \frac{1}{2\pi i} \int_\Omega  \int_\Gamma e^{zt}\delta(\bx-\bx_m) \hat{\phi}(\bx,z) dz d\bx \]
To compute the variation $\delta I$ with respect to the functions $ s_K$ and $s_S$, we use the standard adjoint field approach. First, take the Laplace transform of Equations~\eqref{eqn:timedomain} and multiply throughout by test functions $\hat\psi(\bx,z)$, then integrate by parts and apply appropriate boundary conditions to get, 
\[ \int_\Omega e^{s_K} \nabla \hphi\cdot  \nabla\hpsi d\bx + z\int_{\partial \Omega_w} S_y \hphi\hpsi d\bx + z \int_\Omega e^{s_S} \hphi\hpsi d\bx = \int_\Omega \hat{q} \hpsi d\bx \]
Taking the variation $\delta \hphi$, we have 
\begin{align}
\nonumber \int_\Omega e^{s_K} \nabla \delta \hphi\cdot  \nabla\hpsi d\bx + z\int_{\partial \Omega_w} S_y \delta \hphi\hpsi d\bx + z \int_\Omega e^{s_S} \delta \hphi\hpsi d\bx  &  \\ = \quad - \int_\Omega \left( e^{s_K} \delta s_K\nabla \hat{\phi}\cdot \nabla \hat\psi+ z e^{s_S}\delta s_S \hat{\phi}\cdot \nabla \psi\right) d\bx \label{eqn:senstemp} 
\end{align}
The adjoint field $\hpsi$ is chosen to satisfy the following set of differential equations, which is similar to the Laplace transformed version of Equations~\eqref{eqn:timedomain} with the forcing term corresponding to a measurement operator at the measurement location $\bx_m$. 
\begin{align}\label{adjoint}
- \nabla \cdot \left(e^{s_K}\nabla \hat{\psi}(\bx,z) \right) + z e^{s_S} \hat{\psi}(\bx,z) \quad = & \quad - \delta (\bx-\bx_m) &\bx & \in \Omega   \\ \nonumber
\hat{\psi}(\bx,z) \quad = & \quad 0, & \bx & \in \partial \Omega_D  \\ \nonumber 
\nabla \hat{\psi}(\bx,z) \cdot {n} \quad  = & \quad 0, & \bx &\in \partial \Omega_N  \\\nonumber 
 e^{s_K}\nabla \hat{\psi}(\bx,z)\cdot {n} = & \quad -z S_y\hpsi  , & \bx &\in \partial \Omega_w
\end{align}
Similarly, we multiply equation~\eqref{adjoint} by $\delta \hphi$ and integrate by parts to obtain, 
\begin{align*}
\int_\Omega e^{s_K} \nabla \delta \hphi\cdot  \nabla\hpsi d\bx + z\int_{\partial \Omega_w} S_y \delta \hphi\hpsi d\bx + z \int_\Omega e^{s_S} \delta \hphi\hpsi d\bx  = \int_\Omega \delta(\bx-\bx_m) \delta \hat{\phi}(x,z) \, dx 
\end{align*}
Equating the right hand sides from the Equation~\eqref{eqn:senstemp} and the equation above, we get 
 \begin{align} \delta I = & \quad \frac{1}{2\pi i} \int_\Gamma\int_\Omega \delta(\bx-\bx_m) \delta \hat{\phi}(x,z) dz \, dx \\
 = & \quad \frac{1}{2\pi i} \int_\Gamma \int_\Omega e^{zt} \left( e^{s_K} \delta s_K \nabla \hat{\phi}\cdot \nabla \hat\psi  + z e^{s_S} \delta s_S \hat{\phi}\cdot \hat\psi \right)dz d\bx
 \end{align} 
Finally, the integral w.r.t $z$ can be discretized using the same contour integral approach described in Section~\ref{sec:contour}  
\begin{align} 
\delta I_{N_z} =  & \quad \sum_{k=1}^{N_z} w_k\int_\Omega  \left( e^{s_K} \delta s_K \nabla \hat{\phi}_k\cdot \nabla \hat\psi_k + e^{s_S}\delta s_S \hat{\phi}_k\cdot  \hat\psi_k \right)\, dx
\end{align}
As an example of this calculation, we plot the sensitivity at $t = 5$ [min] (made dimensionless) with respect to log transmissivity. The field here was chosen to be a constant field, hence the symmetry observed between the measurement location $(40,50)$ and the point source $(50,50)$. The results are shown for time $t = 5$ [min]. This is displayed in Figure~\ref{fig:sens}.

\begin{figure}
\centering
\includegraphics[scale = 0.4]{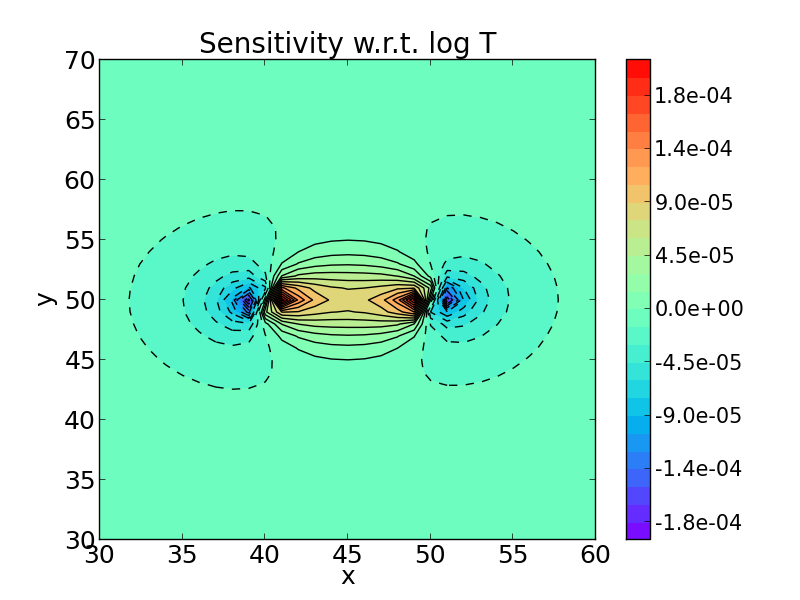}
\caption{Sensitivity (made dimensionless) with respect to log transmissivity. The field here was chosen to be a constant field, hence the symmetry observed between the measurement location $(40,50)$ and the point source $(50,50)$.} \label{fig:sens}
\end{figure}

\subsection{Application: Transient Hydraulic Tomography}
The objective is to reconstruct log transmissivity of an aquifer given measurements of the hydraulic head at several observation locations. The governing equations are the groundwater equations for a $2$D confined aquifer, as described in Section \ref{sec:results}. The `true' field is chosen to be that generated by the exponential kernel as shown in Figure \ref{examplefields} also provided in Equation~\eqref{eqn:exp}. The parameters $R$ and $X$ are chosen to be $R = 10^{-7}I$ and $X = [1,\dots,1]^T$ respectively. We do not investigate the optimality of these parameters, i.e. those that yield the best reconstruction, since the goal was to demonstrate the performance of our proposed solver in solving inverse problems. We additionally added model error by generating measurements using Crank-Nicolson to solve the governing equations using the `true' log transmissivity field. We have three measurement points: at $t = 8, 10$ and $20$ [min] and $36$ measurement locations spaced evenly in the square $[20,80] \times [20,80]$. We introduce a $2\%$ error in the measurements. The inverse problem is solved using these measurements. Storativity was held constant at $S_s = 10^{-5}$ [-] and the domain was a square domain of length $L = 100$ [m]. The single pumping source, located at $(50,50)$, has a pumping rate of $0.85$ L/s. We use the flexible Krylov solver with $N_z = 20$ for solving the forward and adjoint problems on a problem with $101 \times 101$ grid points. The results are shown in Figure \ref{fig:inversion}. The error in the relative L$_2$ norm is $0.16$ within the area of measurements, i.e. the $[20,80] \times [20,80]$ m$^2$ box, and the total error in the relative L$_2$ norm is $0.36$ for the whole aquifer.

Figure~\ref{fig:comparison} shows the time required for the various solvers. `Direct' refers to solving each of the shifted systems using a direct solver, and FGMRES-Sh is the solver using $2$ preconditioners per time. $N_z$ was chosen to be $20$. Note that the slope of the graph corresponding to Crank-Nicolson is dependent on the particular application and on the time step used. The efficiency of Crank-Nicolson relative to FGMRES-Sh depends on the time at which the solution is to be evaluated as well as the size of the time steps taken. For our application, the speedup using FGMRES-Sh as opposed to Crank-Nicolson to solve the forward problem was significant, as demonstrated in the figure. The computational gain is more dramatic in the calculation of the Jacobian since this solution is required for multiple sources and receivers.

\begin{figure}
\centering
\includegraphics[scale = 0.35]{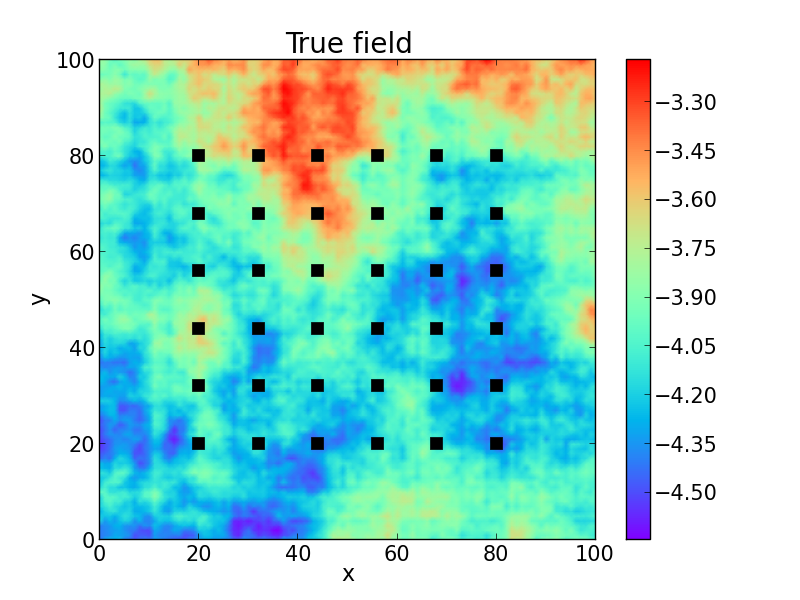}
\includegraphics[scale = 0.35]{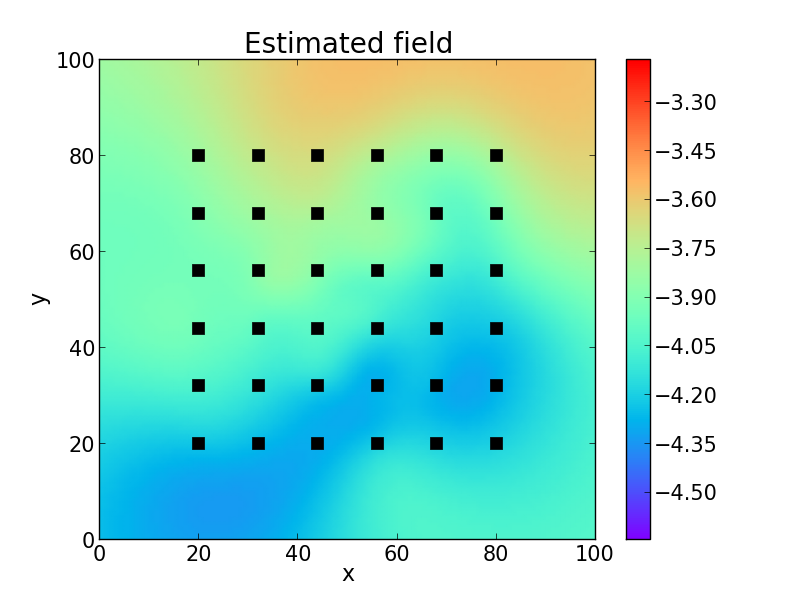}
\caption{True transmissivity field (left) and the estimates transmissivity field (right). The number of grid points is $101,124$. The error in the relative L$_2$ norm is $0.11$ within the area of measurements, i.e. the $[20,80] \times [20,80]$ m$^2$ box, and the total error in the relative L$_2$ norm is $0.35$ for the whole aquifer.
} \label{fig:inversion}
\end{figure}

\begin{figure}
\centering
\includegraphics[scale = 0.4]{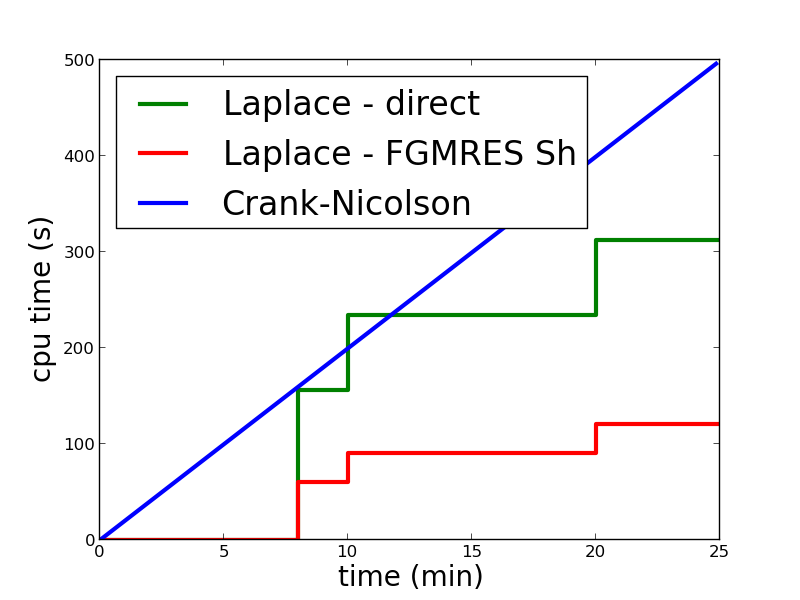}
\caption{Comparison between Crank-Nicolson and the Laplace Transform method using direct solver and GMRES for shifted systems. Note that the slope of the line corresponding to Crank-Nicolson depends on the time step that is used.} \label{fig:comparison}
\end{figure}

Let us assume that we have $N_s$ sources, $N_m$ receivers and collect $N_T$ time measurements. The total number of systems that need to be solved are $N_sN_mN_T$. Assuming the application of the preconditioner can be represented by $\mu(N)$ and the number of iterations using the Krylov solvers are $N_\text{iter}$. The total cost of computing the Jacobian is $N_mN_sN_\text{iter}N_T\mu(N)$. Using a time-stepping scheme with $N_t$ time steps (and therefore, the solution of $N_t$ system of equations) similarly costs $N_mN_sN_t\mu(N)$. The costs are comparable when $N_t \sim N_{iter}N_T$. In Section~\ref{sec:costs} we that for increasing accuracy, the number of systems to be solved increase only logarithmically for Laplace transform-based methods whereas it has a square-root growth for time-stepping schemes. For a few time measurement points $N_T$ or when the number of time steps required by the time-stepping schemes is large, we expect that the Laplace transform-based methods are much more efficient. This is also confirmed by Figure~\ref{fig:comparison}.

\section{Conclusion and future work}\label{sec:conc}

We have proposed a fast method for solving large-scale nonlinear inverse problems based on parabolic PDE. While we have focused on THT, it can be applied to a general class of time-dependent PDEs, for which the Laplace transform can be applied. The resulting system of equations are solved efficiently using a flexible Krylov approach, previously used to accelerate oscillatory hydraulic tomography. For small number of measurement times, our solver is computationally more efficient than standard time-stepping schemes, especially when very small time steps required for stability. We have applied the solver to synthetic problems arising in THT. Since the computation of the Jacobian is often the bottleneck in solving large-scale inverse problems, our approach for computing the Jacobian based on the Laplace transform greatly improves the storage and computational cost.

There are several avenues for future work. In our work, the measurement times are considered known a priori and are few in number. However, the number of measurement times and the sampling rate can be chosen to better improve the accuracy of the reconstruction. For example, choosing the sampling points as the roots of Laguerre polynomials may be used to improve the accuracy of the integral corresponding to the data mismatch. Regarding the flexible Krylov solver, theoretic insight into the properties of the rational Krylov subspace generated $U_m$ can better guide the user to pick preconditioners to efficiently precondition the shifted systems, including those corresponding to multiple times. It remains to be demonstrated, if based on our approach, we can develop a solver that can solve for the entire time history. Another computational bottleneck we would like to consider is the solution of the shifted system in Equation~\eqref{eqn:multipleshifted} with multiple right hand sides, corresponding to different source and measurement locations. One may either try block approaches, or recycling strategies for systems with multiple shifts and multiple right hand sides. 

\section{Acknowledgments}

This work was supported by NSF award 1215742, ``Collaborative Research: Fundamental Research on Oscillatory Flow in Hydrogeology." We would also like to thank James Lambers and Anil Damle for their careful reading of the manuscript and useful suggestions. 

\bibliographystyle{plain}
\bibliography{main.bib}

\begin{thebibliography}{10}

\bibitem{boas2001imaging}
D.A. Boas, D.H. Brooks, E.L. Miller, C.A. DiMarzio, M.E. Kilmer, R.J. Gaudette,
  and Q.~Zhang.
\newblock Imaging the body with diffuse optical tomography.
\newblock {\em Signal Processing Magazine, IEEE}, 18(6):57--75, 2001.

\bibitem{cardiff2013hydraulic}
M.~Cardiff, W.~Barrash, and P.K. Kitanidis.
\newblock Hydraulic conductivity imaging from 3-{D} transient hydraulic
  tomography at several pumping/observation densities.
\newblock {\em Water Resources Research}, 49(11):7311--7326, 2013.

\bibitem{druskin2011adaptive}
Vladimir Druskin and Valeria Simoncini.
\newblock Adaptive rational krylov subspaces for large-scale dynamical systems.
\newblock {\em Systems \& Control Letters}, 60(8):546--560, 2011.

\bibitem{franke1979critical}
R.~Franke.
\newblock A critical comparison of some methods for interpolation of scattered
  data.
\newblock Technical report, DTIC Document, 1979.

\bibitem{griewank2000algorithm}
A.~Griewank and A.~Walther.
\newblock Algorithm 799: {REVOLVE}: an implementation of checkpointing for the
  reverse or adjoint mode of computational differentiation.
\newblock {\em ACM Transactions on Mathematical Software (TOMS)}, 26(1):19--45,
  2000.

\bibitem{gu2007flexible}
G.~Gu, X.~Zhou, and L.~Lin.
\newblock A flexible preconditioned {A}rnoldi method for shifted linear
  systems.
\newblock {\em Journal of Computational Mathematics - International Edition},
  25(5):522, 2007.

\bibitem{hager1984condition}
W.W. Hager.
\newblock Condition estimates.
\newblock {\em SIAM Journal on Scientific and Statistical Computing},
  5(2):311--316, 1984.

\bibitem{higham2000block}
N.J. Higham and F.~Tisseur.
\newblock A block algorithm for matrix 1-norm estimation, with an application
  to 1-norm pseudospectra.
\newblock {\em SIAM Journal on Matrix Analysis and Applications},
  21(4):1185--1201, 2000.

\bibitem{in2011contour}
K.J. In't~Hout and J.A.C Weideman.
\newblock A contour integral method for the {B}lack-{S}choles and {H}eston
  equations.
\newblock {\em SIAM Journal on Scientific Computing}, 33(2):763--785, 2011.

\bibitem{kaipio2006statistical}
Jari Kaipio and Erkki Somersalo.
\newblock {\em Statistical and computational inverse problems}, volume 160.
\newblock Springer Science \& Business Media, 2006.

\bibitem{kilmer2006recycling}
M.E. Kilmer and E.~de~Sturler.
\newblock Recycling subspace information for diffuse optical tomography.
\newblock {\em SIAM Journal on Scientific Computing}, 27(6):2140--2166, 2006.

\bibitem{kitanidis1995quasi}
P.~K. Kitanidis.
\newblock Quasilinear geostatistical theory for inversing.
\newblock {\em Water Resour. Res.}, 31(10):2411--2419, 1995.

\bibitem{kitanidis2010bayesian}
P.~K. Kitanidis.
\newblock {\em Bayesian and Geostatistical Approaches to Inverse Problems},
  pages 71--85.
\newblock John Wiley \& Sons, Ltd, 2010.

\bibitem{LoggMardalEtAl2012a}
A.~Logg, K.A.M., G.N. Wells, et~al.
\newblock {\em Automated Solution of Differential Equations by the Finite
  Element Method}.
\newblock Springer, 2012.

\bibitem{LoggWells2010a}
A.~Logg and G.N. Wells.
\newblock Dolfin: Automated finite element computing.
\newblock {\em ACM Transactions on Mathematical Software}, 37(2), 2010.

\bibitem{LoggWellsEtAl2012a}
A.~Logg, G.N. Wells, and J.~Hake.
\newblock {\em DOLFIN: a C++/Python Finite Element Library}, chapter~10.
\newblock Springer, 2012.

\bibitem{lopez2006analysis}
Luciano Lopez and Valeria Simoncini.
\newblock Analysis of projection methods for rational function approximation to
  the matrix exponential.
\newblock {\em SIAM Journal on Numerical Analysis}, 44(2):613--635, 2006.

\bibitem{neuman1972theory}
Shlomo~P Neuman.
\newblock Theory of flow in unconfined aquifers considering delayed response of
  the water table.
\newblock {\em Water Resources Research}, 8(4):1031--1045, 1972.

\bibitem{popolizio2008acceleration}
Marina Popolizio and V~Simoncini.
\newblock Acceleration techniques for approximating the matrix exponential
  operator.
\newblock {\em SIAM Journal on Matrix Analysis and Applications},
  30(2):657--683, 2008.

\bibitem{ruhe1984rational}
A.~Ruhe.
\newblock Rational {K}rylov sequence methods for eigenvalue computation.
\newblock {\em Linear Algebra and its Applications}, 58:391--405, 1984.

\bibitem{saad2003iterative}
Yousef Saad.
\newblock {\em Iterative methods for sparse linear systems}.
\newblock Siam, 2003.

\bibitem{saibaba2013flexible}
A.K. Saibaba, T.~Bakhos, and P.K. Kitanidis.
\newblock A flexible {K}rylov solver for shifted systems with application to
  oscillatory hydraulic tomography.
\newblock {\em SIAM Journal on Scientific Computing}, 35(6):A3001--A3023, 2013.

\bibitem{saibaba2012efficient}
A.K. Saibaba and P.K. Kitanidis.
\newblock Efficient methods for large-scale linear inversion using a
  geostatistical approach.
\newblock {\em Water Resources Research}, 48(5), 2012.

\bibitem{sheen2000parallel}
D.~Sheen, I.H. Sloan, and V.~Thom{\'e}e.
\newblock A parallel method for time-discretization of parabolic problems based
  on contour integral representation and quadrature.
\newblock {\em Mathematics of Computation of the American Mathematical
  Society}, 69(229):177--195, 2000.

\bibitem{sheen2003parallel}
D~Sheen, I.H. Sloan, and V.~Thom{\'e}e.
\newblock A parallel method for time discretization of parabolic equations
  based on {L}aplace transformation and quadrature.
\newblock {\em IMA Journal of Numerical Analysis}, 23(2):269--299, 2003.

\bibitem{simoncini2007recent}
V.~Simoncini and D.B. Szyld.
\newblock Recent computational developments in {K}rylov subspace methods for
  linear systems.
\newblock {\em Numerical Linear Algebra with Applications}, 14(1):1--59, 2007.

\bibitem{symes2007reverse}
W.W. Symes.
\newblock Reverse time migration with optimal checkpointing.
\newblock {\em Geophysics}, 72(5):SM213--SM221, 2007.

\bibitem{tarantola2005inverse}
Albert Tarantola.
\newblock {\em Inverse problem theory and methods for model parameter
  estimation}.
\newblock siam, 2005.

\bibitem{thomee2005high}
V.~Thom{\'e}e.
\newblock A high order parallel method for time discretization of parabolic
  type equations based on {L}aplace transformation and quadrature.
\newblock {\em Int. J. Numer. Anal. Model}, 2:121--139, 2005.

\bibitem{trefethen2006talbot}
L.N. Trefethen, J.A.C. Weideman, and T.~Schmelzer.
\newblock Talbot quadratures and rational approximations.
\newblock {\em BIT Numerical Mathematics}, 46(3):653--670, 2006.

\bibitem{wang2009minimal}
Q.~Wang, P.~Moin, and G.~Iaccarino.
\newblock Minimal repetition dynamic checkpointing algorithm for unsteady
  adjoint calculation.
\newblock {\em SIAM Journal on Scientific Computing}, 31(4):2549--2567, 2009.

\bibitem{weideman2007parabolic}
J~Weideman and L~Trefethen.
\newblock Parabolic and hyperbolic contours for computing the {B}romwich
  integral.
\newblock {\em Mathematics of Computation}, 76(259):1341--1356, 2007.

\bibitem{weideman2006optimizing}
J.A.C. Weideman.
\newblock Optimizing {T}albot's contours for the inversion of the {L}aplace
  transform.
\newblock {\em SIAM Journal on Numerical Analysis}, 44(6):2342--2362, 2006.

\bibitem{rasmussen2006gaussian}
C.K.I. Williams and C.E. Rasmussen.
\newblock Gaussian processes for machine learning.
\newblock {\em the MIT Press}, 2(3):4, 2006.

\bibitem{zaslavsky2012large}
M~Zaslavsky, V~Druskin, A~Abubakar, T~Habashy, and V~Simoncini.
\newblock Large-scale {G}auss-{N}ewton inversion of transient {CSEM} data using
  the model reduction framework.
\newblock {\em Geophysics}, 2012.

\end{thebibliography}

\end{document}